%% file: sector-template3-2019-07-09-ex2.tex
\RequirePackage{fix-cm}
\documentclass[smallextended]{svjour3}       
\smartqed  
\usepackage{graphicx}

\input my-macro-no-
\def\deg{^\circ}
\def\30{$30\deg$}
\def\60{$60\deg$}
\def\a{\angle}

\begin{document}

\title{Wetzel's sector covers unit arcs
}


\author{Chatchawan Panraksa \and  Wacharin Wichiramala 
}


\institute{Chatchawan Panraksa \at
              Applied Mathematics Program, Science Division, Mahidol University International College, 999 Phutthamonthon 4 Road, Salaya, Nakhonpathom, Thailand 73170 \\
              \email{chatchawan.pan@mahidol.edu}           
           \and
           Wacharin Wichiramala \at
              Department of Mathematics and Computer Science, Faculty of Science, Chulalongkorn University, 254 Phayathai Road, Pathumwan, Bangkok, Thailand 10330 \\
              \email{wacharin.w@chula.ac.th}
}
\date{}

\maketitle

\begin{abstract}
We settle J. Wetzel's 1970's conjecture and show that a \30
circular sector of unit radius can accommodate every planar arc of unit
length. \ Leo Moser asked in 1966 for the smallest (convex) region in the
plane that can accommodate each arc of unit length. \ With area $\pi/12,$ this
sector is the smallest such set presently known. \ Moser's question has
prompted a multitude of papers on related problems over the past 50 years,
most remaining unanswered.
\end{abstract}

\section{Introduction}
\label{intro}
On a list \cite[Problem 9]{Leo Moser 1966} of 50 combinatorial geometry
problems circulated in 1966, L. Moser asked for the smallest set in the plane
that is large enough to accommodate a congruent copy of each arc of unit
length. \ (The list and much related material and commentary are included in
W. O. J. Moser \cite[p. 211, pp. 218-19]{W O J Moser 1991}.) \ This difficult
question, still unanswered, has become the model for a family of arc-covering
questions known collectively as \textquotedblleft worm
problems.\textquotedblright\ \ Many articles have been published on such
questions, but very few specific problems have been completely solved.

Let $\mathcal{F}$ be the set of all arcs of unit length in the plane. \ A
\textit{cover} of $\mathcal{F}$ is a set that contains a congruent copy of every
arc in $\mathcal{F}.$ \ Soon after Moser asked his question, A. Meir showed
that a semidisk of unit diameter is such a cover (reported in \cite{Wetzel
1973}); and in 1973 Wetzel \cite{Wetzel 1973} showed more generally that for
$0%
{{}^\circ}%
<\vartheta\leq90%
{{}^\circ}%
$ a circular sector of angle $\vartheta$ and\ radius $\frac{1}{2}\csc\frac
{1}{2}\vartheta$ is also a cover. \ At about this same time, Wetzel guessed
that the $30%
{{}^\circ}%
$ sector $\Pi$ of unit radius might also be a cover, but although he mentioned
this hunch to colleagues, he inexplicably neglected to mention it in
\cite{Wetzel 1973}. \ It was first reported in \cite{Norwood et al 1992} and
more recently in \cite[p. 358]{Wetzel 2003}.

The relevant literature on related covering problems is sizeable, but not much has been published about Moser's specific question. \ An argument is given in \cite{Wetzel-Wichiramala 2018} to show that the sector with central angle $\vartheta$ and radius $\csc2\vartheta$ is a cover, so in
particular, a $30%
{{}^\circ}%
$ sector of radius $2/\sqrt{3}\approx1.155$ is large enough to cover
$\mathcal{F}$. \ It is known \cite{Movshovich-Wetzel 2017} that at least every \textit{convex} (and consequently every \textit{drapeable}) unit arc fits in $\Pi$.

The best bounds in the literature for the least area $\alpha$ of a convex cover for $\mathcal{F}$ are $0.2322<\alpha<0.2709,$ as reported in \cite{Wetzel 2013}. \ With area
$\pi/12\approx0.2618,$ the $30%
{{}^\circ}%
$ unit sector $\Pi$ reduces the upper bound by a little more than 3 percent.

Here we give a geometric proof of Wetzel's conjecture. \ The shape of $\Pi$
suggests that reflections might be helpful, and our proof relies heavily on
this insight. \ But in one of those strange coincidences that sometimes occur,
Y. Movshovich has recently found an altogether different analytic proof of the result
using calculus-level tools. \ Her argument has not
yet been published.

\section{Preliminaries}
\label{sec:2}
In \cite[Cor. 5]{Wetzel-Wichiramala 2010} we showed in detail that a compact
set in the plane is a cover for $\mathcal{F}$ if it contains a congruent copy
of each simple polygonal arc of unit length. 
We first mention the key lemma on how a pair of support lines touches a simple arc. The parallel case is discussed in \cite{Coulton-Movshovich 2006} and the general case is in \cite{Wetzel-Wichiramala 2018}. Here we give another proof for the general case.
This is called the \textit{lambda property}. Our proof of the following lemma is related to reasoning given by J. Ralph Alexander in our unpublished paper, Alexander, Wetzel, and Wichiramala, ``The $\Lambda$-property of a simple arc." \cite{aww}.
The proof of the next lemma will appear after the main theorem.
\begin{lem}
Suppose $\gamma$ is a parametrized, simple, polygonal arc and $0<\theta<\pi$.
Let $X_1, X_2, ...,X_n$ be the corner points of the convex hull of $\gamma$ and appear on $\gamma$ in this parametric order .
Then $\gamma$ can be placed within two angles of size $\theta$ so that

1) $X_1$ touches an angle's vertex or $\gamma$ touches the rays at 3 points with parameters
$t_1<t_2<t_3$ where $\gamma(t_2)$ touches one ray and the other two points
touch the other ray.

2) $X_n$ touches the other angle's vertex or $\gamma$ touches the rays at another set of 3 alternating points. 

In total, there are at most 2 pairs of support lines with sets of such triple points. Furthermore when there are 2 pairs of support lines, they interlace as illustrated in \figref{Fig4} or \figref{Fig5}.
\end{lem}

The lemma originates in \cite{Wichiramala 2013} with the same result except the obvious interlace property. We note that 2 lines make angle 0 if they are parallel.
From the previous lemma, when $\theta=0$ the 2 results 1) and 2) become existing
of a unique parallel support lines with such 3 alternating points \cite{Wichiramala 2013}. When $\theta\ge\pi$,
the result is trivial and thus not interesting.
In addition, we may place some $\gamma$ within an angle so that some middle $X_i$, $1<i<n$, touches the angle's vertex. 
When $\theta$ is equal to the angle of the
hull at $X_1$, we can both place $\gamma$ within an angle of size $\theta$ so that $X_1$ touches the angle's vertex
and there is a set of 3 alternating points having $X_1$ and $X_2$ (on opposite
rays).
Recently, Movshovich extends the result in \cite{Movshovich 2019}.

The next lemma shows how to compare the length of an arbitrary arc with the chord of a related sector.

\begin{lem}\label{lem:chords}
Let $AVB$ be a sector of angle $\theta$ with chord $AB$. Suppose that the arc $\gamma$ with endpoints $P$ and $Q$ are on the rays $VA$ and $VB$, respectively (see \figref{Fig2}).
\figs{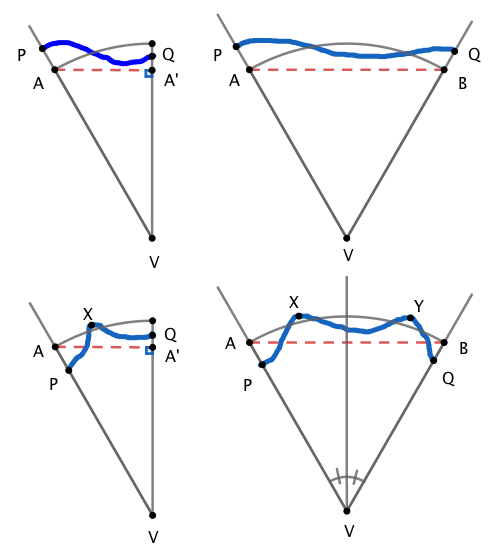}{Length comparison between arcs and chords.}{.4}

1) If $P$ is not in the sector and $\theta<90\deg$, then $PQ>AA'$.

2) If $P$ and $Q$ are not in the sector (one of them is allowed to be on the circular arc) and $\theta<180\deg$, then $PQ>AB$.

3) If a point $X$ on $\gamma$ is not in the sector and $\theta<90\deg$, then $\gamma$ is longer than $AA'$. In addition, if $X$ is on the angle bisector, then $\gamma$ is longer than $AB$.

4) If points $X$ and $Y$ on $\gamma$ are not in the sector and they are on opposite side of the angle bisector of $\theta<180\deg$, then $\gamma$ is longer than $AB$.
\end{lem}
\begin{proof}
This is clear by simple comparison and the inequality $\sin\alpha+\sin\beta > \sin\alpha\cos\beta+\cos\alpha\sin\beta = \sin(\alpha+\beta)$ for $\alpha$ and $\beta$ in $(0,90\deg)$.
\qed
\end{proof}

\section{The main result}
\label{sec:3}
We prove the conjecture on the sector $\Pi$ here.

\begin{thm}
A 30$\deg$ sector of unit radius contains a congruent copy of every unit
arc.
\end{thm}

\begin{proof}
It suffices to show that every simple, polygonal unit
arc fits in $\Pi$. Suppose to the contrary that $\gamma$ is a simple,
polygonal unit arc that cannot be covered by $\Pi$ (see Corollary 5 in \cite{Wetzel-Wichiramala 2010}).
Hence $\gamma$ cannot be placed within $\Pi$ so that it touches the angle's vertex of the sector of unit radius.
By Lemma 1 for angle $\theta=30\deg$, there are 2 different pairs of support lines with angle in between $30\deg$ together with 2 sets of 3 alternating points on each pair of such lines. Some of these 6 points may coincide. Let $U$ and $V$ be the intersections of the 2 pairs of support lines
(see \figref{Fig3}).
\figs{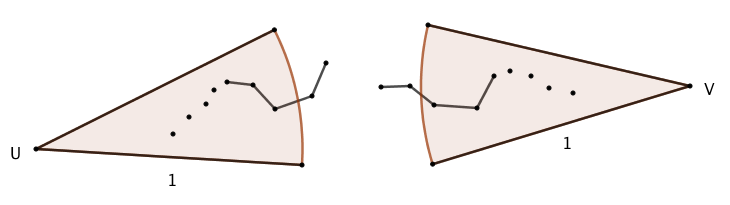}{The 2 angles at vertices $U$ and $V$ of size $30\deg$.}{.4}
Hence none of the 6 points touches $U$ or $V$,
and there are points $X$ and $Y$ on $\gamma$ that $XV>1$ and $YU>1$.

We will divide into cases according to relations between the 2 pairs of lines and the 2 sets of points.

Case 1. The 2 pairs of lines share a common line with 2 sets of points $ABD$ and $ACD$ as in \figref{Fig4}.
\figs{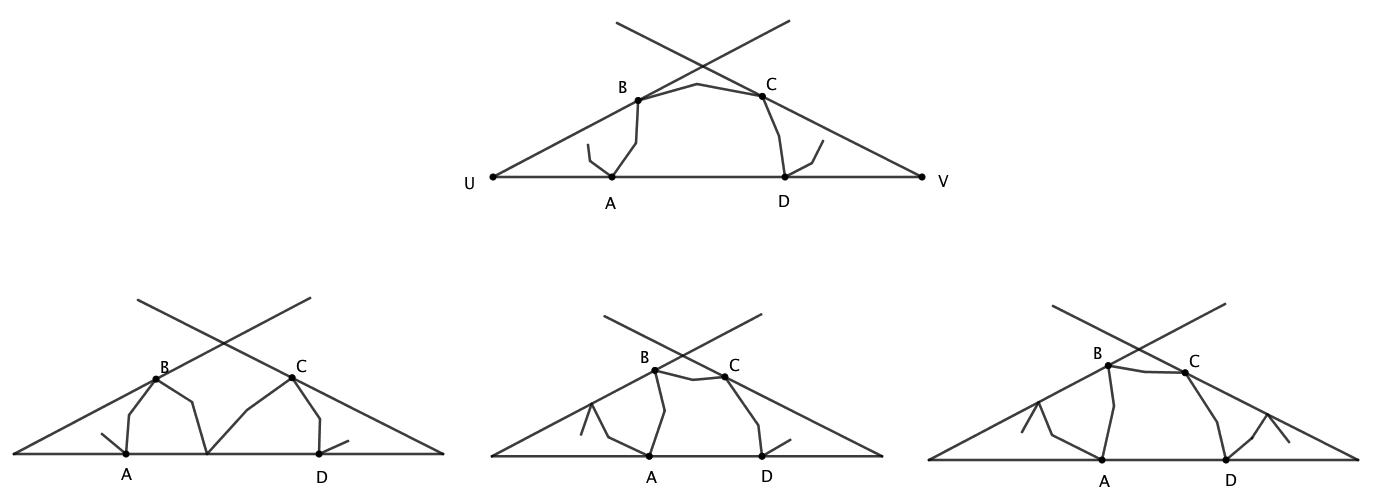}{Case 1 with all possibilities}{.25}
Note that $B$ and $C$ may coincide as a degenerate subcase. This case already includes the last 3 possibilities in \figref{Fig4} as we may ignore some points.

For the remaining cases, the 4 lines are different. 
The arrangement of those 6 points will determine the remaining cases.

Case 2. The 2 sets of points are $ABD$ and $CEF$ as in \figref{Fig5}(a).
\figs{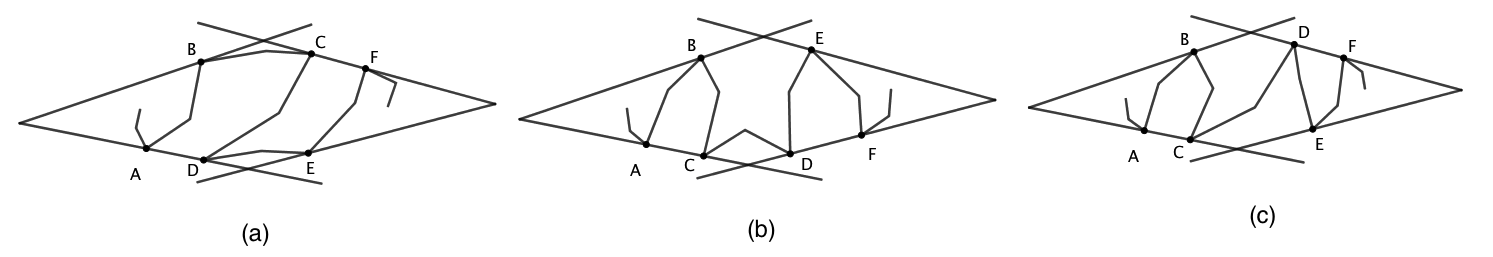}{(a) Case 2. (b) Case 3. (c) Case 4.}{.25}
Note that $BC$ or $DE$ may degenerate.

Case 3. The 2 sets of points are $ABC$ and $DEF$ as in \figref{Fig5}(b).
Note that $CD$ may degenerate.

Case 4. The 2 sets of points are $ABC$ and $DEF$ as in \figref{Fig5}(c).

Let $P$ be the endpoint of $\gamma$ near $A$ and $Q$ be the other endpoint.
Now we will show in each case that $\gamma$ is longer than 1.



Case 1. Here we have that $X$ is on the subarc $PB$ and $Y$ is on the subarc $CQ$.
First suppose that $X$ is on the subarc $AB$ and $Y$ is on the subarc $CD$ (see \figref{Fig6}).
\figs{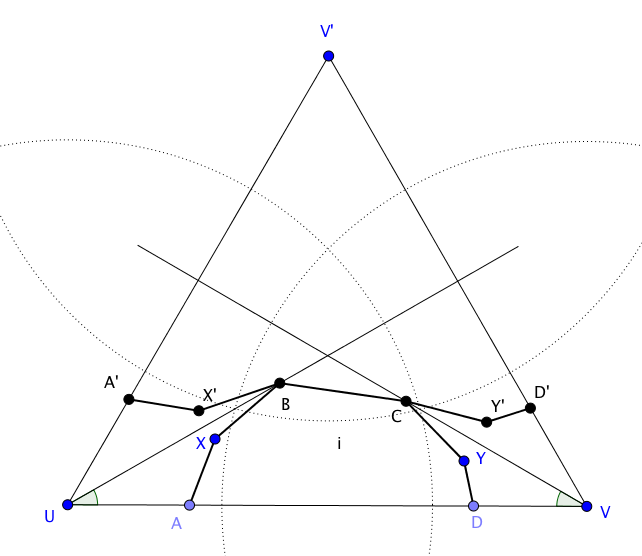}{Case 1.}{.3}
Note that $X$ may coincide with $A$ or $B$ and $Y$ may coincide with $C$ or $D$. We will show that the polysegment $AXBCYD$ is longer than 1.
First we reflect $A$, $X$ and $V$ across $UB$ to $A'$, $X'$ and $V'$. Then we reflect $Y$, $D$ and $U$ across $VC$ to $Y'$, $D'$ and $V'$. Since $X'V'=XV>1$ and $V'Y'=UY>1$, by \lemref{chords}, The length $AXBCYD=A'X'BCY'D'$ is greater than the chord of a unit sector of angle $60\deg$ at $V'$. Therefore $\gamma$ is longer than 1.

Now suppose that $X$ is on the subarc $PA$ (and $Y$ is on the subarc $EF$). We will compare the polygonal arc $XABCDEYF$ with a shorter polygonal arc that satisfies the previous assumption. First let $\bar{X}$ be the point on the left of $A$ with $\bar{X}A=XA$ (see \figref{Fig7}).
\figs{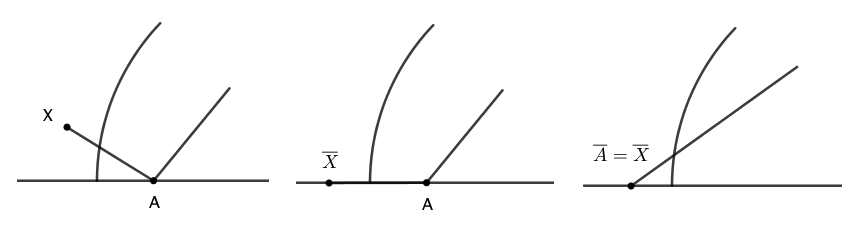}{$X$ is on the subarc $PA$. ($\bar{X}A=XA$)}{.3}
Since $A$ is on the left of $V$, the center of the circular arc, $\bar{X}$ is also on the left of the circular arc. Thus $\bar{X}V>1$. Now let $\bar{A}=\bar{X}$. Thus the polygonal arc $\bar{A}\bar{X}BCDEYF$ is shorter than $XABCDEYF$. By the first part of case 1, we have $\bar{A}\bar{X}BCDEYF>1$.
When $Y$ is on the subarc $FQ$, the argument is similar. Hence $\gamma$ is longer than 1.

For the remaining cases, with the same argument on $\bar{X}$ and $\bar{A}$, we may suppose that $X$ is on the subarc $AB$ and $Y$ is on the subarc $EF$. Then we will show that the polysegment $AXBCDEYF$ is longer than 1.

\figs{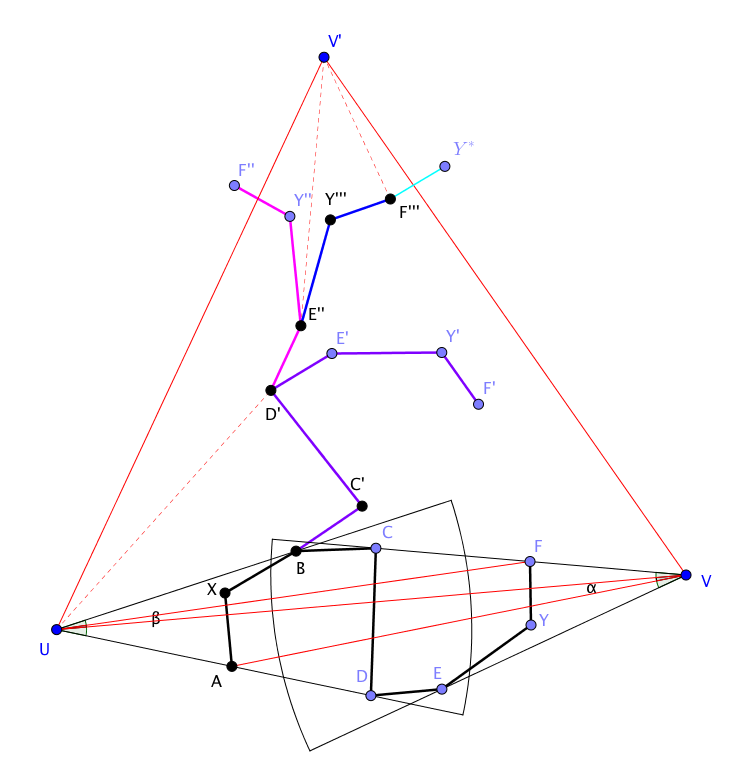}{Case 2}{.3}
Case 2.
Let $\alpha = \a AVU$ and $\beta = \a FUV$ (see \figref{Fig8}). Possibly after a rotation, we may assume $\alpha \ge \beta$.
Note that the product of 2 reflections in the sides of a \30 angle is a \60 rotation about their intersection.
We create new points and segments as follows.
First let $C'D'E'Y'F'$ be the reflection of $CDEYF$ across $UB$.
Next let $E''Y''F''$ be the reflection of $E'Y'F'$ across $UD'$.
Now let $V'$ be the rotation of $V$ for \60 around $U$.
Then let $Y'''F'''$ be the reflection of $Y''F''$ across $V'E''$ and $Y^*$ be the reflection of $Y'''$ across $V'F'''$.
Note that $V$ is also the rotation of $U$ for \60 around $V'$.

Since $\a F'''V'V = \a FVU$, it is in $(0,30\deg)$. We then have
$\a Y'''V'V = \a Y'''V'F''' + \a F'''V'V > \a Y'''V'F''' - \a F'''V'V = \a Y^*V'F''' - \a F'''V'V = \a Y^*V'V$. 
Hence $V$ and $Y^*$ are on the same side of $V'F'''$. 
Thus $Y'''V>Y^*V=YU>1$. 
Together with $XV>1$ and $\a F'''VV' = \a FUV=\beta$. By \lemref{chords}, the length of $AXBC'D'E''Y'''F'''$ is at least $AF'''$ which is longer than a chord of a unit sector around $V$ of angle $60\deg+\alpha-\beta \ge 60\deg$. Therefore the length of $AXBCDEYF$ is greater than 1. In degenerate case where $B=C$ or $D=E$, the argument is similar.

\figs{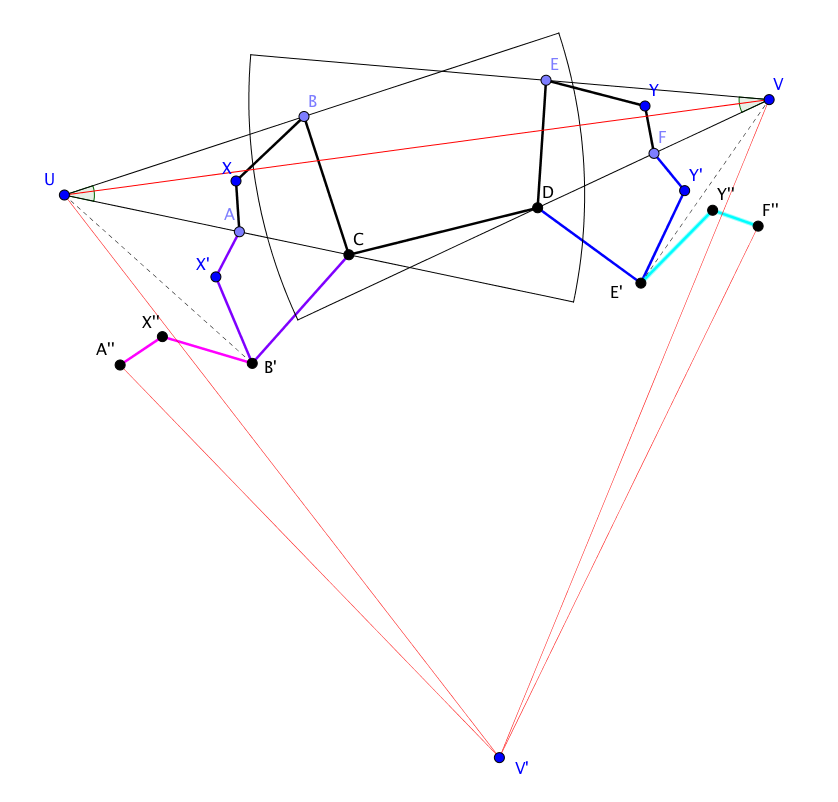}{Case 3}{.3}
Case 3.
We create new points and segments as follows (see \figref{Fig9}).
First let $X'B'$ be the reflection of $XB$ across $UA$ and $A''X''$ be the reflection of $AX'$ across $UB'$.
Next let $V'$ be the rotation of $V$ around $U$ for $-60\deg$.
Now let $E'Y'$ be the reflection of $EY$ across $VF$ and $Y''F''$ be the reflection of $Y'F$ across $VE'$.
Note that $V'$ is also the rotation of $U$ around $V$ for \60.

Note that $X''V'=XV>1$ and $Y''V'=YU>1$. By \lemref{chords}, the length of $A''X''B'CDE'Y''F''$ is at least $A''F''$ which is longer than a chord of a unit sector around $V'$ of angle at least \60. Therefore the length of $AXBCDEYF$ is greater than 1.

\figs{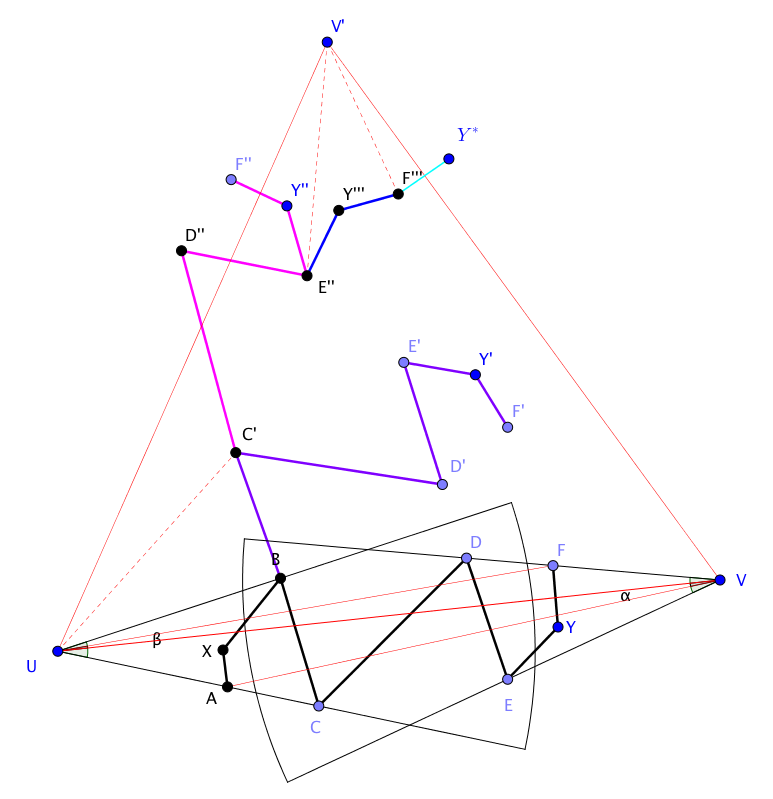}{Case 4}{.3}
Case 4.
Similar to case 2, let $\alpha = \a AVU$ and $\beta = \a FUV$ (see \figref{Fig10}). Possibly after a rotation, we may assume $\alpha \ge \beta$.
We create new points and segments as follows.
First let $C'D'E'Y'F'$ be the reflection of $CDEYF$ across $UB$.
Next let  $D''E''Y''F''$ be the reflection of $D'E'Y'F'$ across $UC'$.
Now let $V'$ be the rotation of $V$ around $U$ for \60.
Then let $Y'''F'''$ be the reflection of $Y''F''$ across $V'E''$ and $Y^*$ be the reflection of $Y'''$ across $V'F'''$.
Note that $V$ is also the rotation of $U$ around $V'$ for \60.

Since $\a Y'''V'V > \a Y^*V'V$, we have $Y'''V>Y^*V=YU>1$. 
Together with $XV>1$ and $\a F'''VV' = \a FUV=\beta$, by \lemref{chords}, the length of $AXBC'D''E''Y'''F'''$ is at least $AF'''$ which is longer than a chord of a unit sector around $V$ of angle $60\deg+\alpha-\beta \ge 60\deg$. Therefore the length of $AXBCDEYF$ is greater than 1.

In each case, we arrive at a contradiction. Therefore $\Pi$ can cover every unit arc.
\qed
\end{proof}

Now we provide the proof of Lemma 1.

\begin{proof}
We first consider the angle of the convex hull of $\gamma$ at $X_1$. If the
angle is not bigger than $\theta$, then we can place $\gamma$ within it so
that $X_1$  touches the angle's vertex. So we suppose the angle is bigger
than $\theta$.
\figs{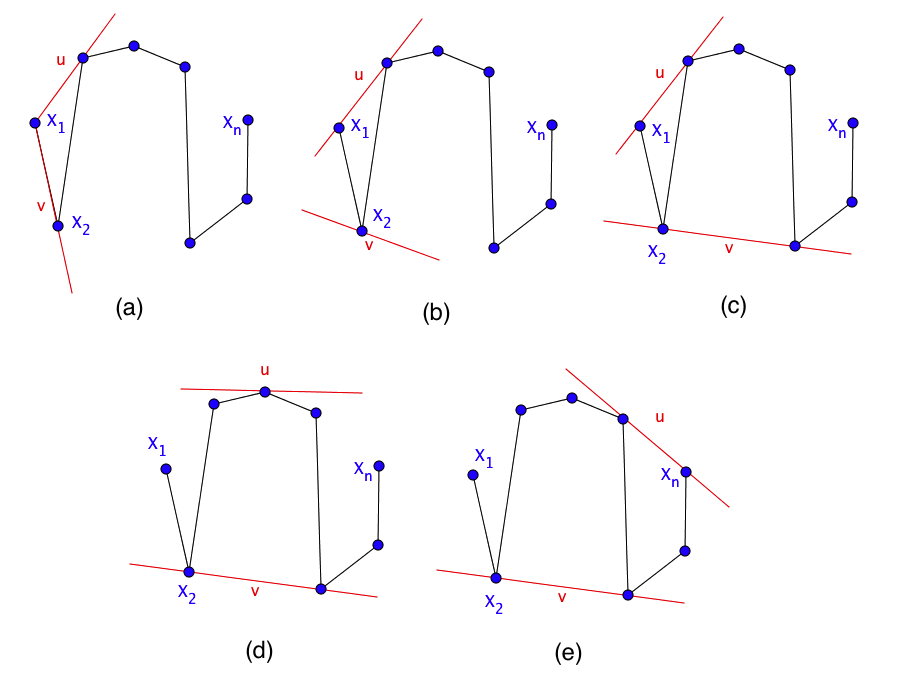}{
(a) The 2 rays $u$ and $v$ start at $X_1$.
(b) The ray $v$ rotates around $\gamma$.
(c) The ray $v$ touches 2 external pieces.
(d) The ray $u$ rotates around $\gamma$.
(e) The ray $u$ touches 2 external pieces.
}{.3}
Note that for each $i$, the segment $X_iX_{i+1}$ is a part of the boundary
of the hull or cross segment (inside the interior of) the hull. Without the
interior
of all cross segments, the remaining of the polygonal arc $X_1 X_2\ldots
X_n$ is composed of connected,
external
pieces on the boundary. Note that an external piece could possibly be just a single
point. However the beginning and ending pieces are not a single point.
Next we place 2 rays $u$ and $v$ starting at $X_1$ and on the 2 sides of
the hull as in \figref{Fig1}(a).
As $X_1$ is the first point on the boundary, we may assume that $v$ touches
$X_2$
and $u$ touches 2 external pieces or 2 ends of the single external piece.
We now start to rotate $v$ around and keep touching $\gamma$ as in \figref{Fig1}(b).
We continue to rotate until $v$ touches 2 external pieces as in \figref{Fig1}(c).
Next we start to rotate $u$ around $\gamma$ as in \figref{Fig1}(d).
We continue to rotate until $u$ touches 2 external pieces as in \figref{Fig1}(e).
Along this rotation, we keep the property of having 3 points alternating
between $u$ and $v$ and reduce the angle between  $u$ and $v$. The process
stops when the angle between $u$ and $v$ is $\theta$.

For result 2), we work in the same fashion at $X_{n}$ and get another pair
of support lines with alternating 3 points. Since $\theta>0$, if there are
2 pairs of support lines, they are different as the 2 orientated angles between
$u$'s and $v$'s are different. In the other way, we may think that we continue
moving $u$ and $v$ starting from $X_{1}$ and decreasing the angle in between
until $u$ and $v$ meet at $X_{n}$. When the signed angle in between is $-\theta$,
we have a pair of support line for 2). We have that the number of external
pieces
is greater than the number of cross segments by 1 and each
external piece alternately contribute to opposite side of the hull with respect
to
$X_1$ and $X_n$. Hence the numbers of external
pieces on both sides 
differ by at most 1. Thus the process must stop at $X_n$ and if there are
2 such pairs of support lines, they must interlace as in \figref{Fig4} or
\figref{Fig5}.

For a pair of support lines with this property, one line must touch 2 consecutive
external
pieces while the other line touches the external piece ordered in between
the 2 previous and opposite pieces. Hence the process has found all possible
such pairs.
This completes the proof of the lemma   .
\qed
\end{proof}

\begin{acknowledgements}
We are grateful to John E. Wetzel for providing the introductory material and the associated references,
 invaluable comments and guidance. Communication with Yevgenya Movshovich and her comments are very appreciated. We also appreciate the help from Banyat Sroysang and Pongbunthit Tonpho. This study is partially supported by the 90th Anniversary of Chulalongkorn University Fund (Ratchadaphiseksomphot Endowment Fund).
\end{acknowledgements}



\end{document}

%% file: my-macro-no-.tex
\def\emph#1{\textbf{\textit{#1}}}

\newtheorem{thm}{Theorem}

\newtheorem{lem}[thm]{Lemma}

\def\lemref#1{Lemma \ref{lem:#1}}

\def\figref#1{Figure \ref{fig:#1}}

\def\figs#1#2#3{%
 \begin{figure}[tp]%
  \centerline{\includegraphics[scale=#3]{#1}}%
  \caption{#2}%
  \label{fig:#1}%
 \end{figure}
 }

\def\real{\mathbb{R}}
\def\R2{\real^2}

\def\.{\ldots}

\def\2g{two-sided component}
\def\3g{three-sided component}
\def\4g{four-sided component}
\def\5g{five-sided~component}
\def\6g{six-sided component}

\def\deg{^\circ}
\def\120{$120\deg$}
\def\90{$90\deg$}